\documentclass[12pt, 
]{amsart}

\usepackage{amsmath}
\usepackage{amsthm}
\usepackage{amssymb}
\usepackage{amscd}
\usepackage{graphicx}
\usepackage[all]{xy}

\numberwithin{equation}{section}

\newtheorem{theo}{Theorem}[section]
\newtheorem{prop}[theo]{Proposition}

\newtheorem{cor}[theo]{Corollary}
\newtheorem{claim}[theo]{Claim}

\theoremstyle{definition}
\newtheorem{defi}[theo]{Definition}
\theoremstyle{remark}
\newtheorem{rem}[theo]{Remark}
\theoremstyle{definition}

\newcommand{\ddbar}{dd^c}

\newcommand{\dbar}{\overline{\partial}}
\newcommand{\e}{\varepsilon}
\newcommand{\ome}{\widetilde{\omega}}
\newcommand{\I}[1]{\mathcal{I}(#1)}
\newcommand{\mi}{{\rm{min}}}
\newcommand{\rbig}{{\rm{big}}}

\newcommand{\lla}[0]{{\langle\!\langle}}
\newcommand{\rra}[0]{{\rangle\!\rangle}}
\begin{document}

\title[]{A Nadel vanishing theorem \\
for metrics with minimal singularities  \\
on big line bundles}

\author{SHIN-ICHI MATSUMURA}

\address{Mathematical Institute, Tohoku University, 6-3, Aramaki Aza-Aoba, Aoba-ku, Sendai 980-8578, Japan.}

\email{{\tt
mshinichi@m.tohoku.ac.jp, mshinichi0@gmail.com}}

\thanks{Classification AMS 2010: 14F18, 32L10,32L20. }


\maketitle

\begin{abstract}
The purpose of this paper is 
to establish a Nadel vanishing theorem 
for big line bundles 
with multiplier ideal sheaves of singular metrics  
admitting an analytic Zariski decomposition 
(such as, metrics with minimal singularities and Siu's metrics). 
For this purpose, we apply the theory of harmonic integrals 
and generalize Enoki's proof 
of Koll\'ar's injectivity theorem. 
Moreover we investigate the asymptotic behavior 
of harmonic forms with respect 
to a family of regularized metrics.  
\end{abstract}

\section{Introduction}
The Kodaira vanishing theorem plays 
an important role 
when we approach certain fundamental problems 
in complex geometry    
(for example, asymptotics of linear systems, 
extension problems of holomorphic sections, 
the minimal model program, and so on). 
By using multiplier ideal sheaves 
associated to singular metrics, 
this theorem is generalized to 
the Nadel vanishing theorem (\cite{Nad89}), 
which can be seen as an analytic counterpart  
of the Kawamata-Viehweg vanishing 
theorem in algebraic geometry (\cite{Kaw82}, \cite{Vie82}). 

In this paper, we study 
singular metrics 
admitting an analytic Zariski decomposition 
(such as, 
metrics with minimal singularities and Siu's metrics) 
and a Nadel vanishing theorem for them 
from the viewpoint of the theory of several complex variables and  complex differential geometry.

\begin{theo}[Nadel vanishing theorem, 
\cite{Nad89}, cf.\cite{Dem82}]\label{Nad}
Let $F$ be a big line bundle on a smooth projective variety $X$ and 
$h$ be a $($singular$)$ metric on $F$ with $($strictly$)$ positive curvature. 
Then we have 
\begin{equation*}
H^{i}(X, K_{X} \otimes F \otimes \I{h})= 0
\quad \text{for}\ \text{any}\ i >0. 
\end{equation*}
Here $\I{h}$ denotes the multiplier ideal sheaf 
of the $($singular$)$ metric $h$ and 
$K_{X}$ denotes the canonical bundle of $X$. 
\end{theo}

We mainly handle  
metrics with minimal singularities $h_{\min}$  
and Siu's metrics $h_{\rm{Siu}}$ 
(see Section 2, 3 for the precise definition). 
These metrics satisfy important properties 
(for example they admit an analytic Zariski decomposition), 
thus several authors have studied them 
(see \cite{DEL00}, \cite{Dem}).

The main purpose of this paper is 
to establish a Nadel vanishing theorem for 
$h_{\min}$ and $h_{\rm{Siu}}$. 
When we 
investigate the cohomology groups with coefficients in 
$K_{X} \otimes F \otimes \I{h_{\min}}$ and 
$K_{X} \otimes F \otimes \I{h_{\rm{Siu}}}$, 
we encounter the following difficulties: 
\begin{itemize}
\item[(1)] $h_{\min}$ and $h_{\rm{Siu}}$ have non-algebraic (transcendental) singularities in general. 
\item[(2)] $h_{\min}$ and $h_{\rm{Siu}}$  do not have strictly positive curvature except the trivial case. 
\end{itemize}

The proof of Theorem \ref{Nad} 
heavily depends on the assumption 
that the curvature of $h$ is \lq \lq strictly" positive. 
Under this assumption, we can construct solutions of 
the $\dbar$-equation with $L^{2}$-estimates, 
which implies Theorem \ref{Nad}. 
Indeed the theorem does not hold 
even if the curvature of $h$ is semi-positive. 
Nevertheless, we can expect 
that all higher cohomology groups 
with coefficients in  
$K_{X} \otimes F \otimes \I{h_{\min}}$ and 
$K_{X} \otimes F \otimes \I{h_{\rm{Siu}}}$ 
vanish from the special characteristics of $h_{\min}$ and $h_{\rm{Siu}}$.
This is because for a big line bundle $F$ we have already known 
\begin{align*}
H^{i}(X, K_{X} \otimes F \otimes \I{\|F \|}) = 0   
\quad  \text{for}\ \text{any}\ i >0, 
\end{align*}
where $\I{\|F \|}$ is the asymptotic multiplier ideal 
sheaf of $F$ (see \cite{DEL00} for the precise definition). 
The multiplier ideal sheaves $\I{h_{\mi}}$ and 
$\I{h_{\rm{Siu}}}$  
can be seen as 
an analytic counterpart of $\I{\|F \|}$, 
and thus all higher cohomology groups with coefficients in 
$K_{X} \otimes F \otimes \I{h_{\min}}$ and 
$K_{X} \otimes F \otimes \I{h_{\rm{Siu}}}$ 
can be expected to vanish. 
The asymptotic multiplier 
ideal sheaf $\I{\|F \|}$ does not always agree with 
$ \I{h_{\min}}$, 
but it is conjectured that the equality 
$\I{\|F \|} = \I{h_{\min}}$ holds 
if $F$ is big or more generally abundant 
(that is, the numerical dimension $\nu(F)$ agrees with the Kodaira dimension $\kappa(F)$).    
Hence it is interesting 
to study a Nadel vanishing theorem for $\I{h_{\mi}}$ and  $\I{h_{\rm{Siu}}}$.  
This is a natural problem, 
however it has been an open problem. 
The following is the main result of this paper, 
which gives an affirmative answer for this problem. 
\begin{theo}[=Corollary \ref{main-co}.]\label{m-min}
Let $F$ be a big line bundle 
on a smooth projective variety $X$ and 
$h_{\min}$ be a metric with minimal singularities on $F$. 
Then we have 
\begin{equation*}
H^{i}(X, K_{X} \otimes F \otimes \I{h_{\min}}) = 0 
\quad {\text{for}}\ {\text{any}}\ i >0.
\end{equation*}
\end{theo}
This theorem follows from Theorem \ref{main}, 
which asserts that all higher cohomology groups with coefficients in $K_{X} \otimes F \otimes \I{h}$ vanish if the singular metric $h$ with semi-positive curvature 
is less singular than some singular metric with strictly positive curvature.  
From Theorem \ref{main}, 
we can obtain the same conclusion   
for $h_{\rm{Siu}}$ (Corollary \ref{main-co}).

These results  
are closely related to 
the openness conjecture of 
Demailly and Koll$\rm{\acute{a}}$r in \cite{DK01}, 
which is a conjecture on singularities of 
plurisubharmonic (psh for short) functions.
When $F$ is a big line bundle and 
$h_{\min}$ has algebraic singularities, 
we can easily see that  $\I{\|F \|} = \I{h_{\min}}$ holds.
However we want to emphasize that 
$h_{\min}$ does not always have 
algebraic singularities (see Section 4). 

For the proof of the main result,  
we need to take a transcendental approach 
since we need to overcome difficulties (1), (2).  
The proof is based on a combination 
of techniques to solve the $\dbar$-equation 
and Enoki's proof of Koll\'ar's injectivity theorem 
for line bundles admitting smooth metrics with semi-positive curvature 
(see \cite{Eno90}, \cite{Kol86}). 
The strategy of the proof 
can be divided into four steps 
as follows:

In Step 1, 
we approximate the metric $h_{\min}$ 
by singular metrics 
$\{ h_{\e} \}_{\e>0}$ that are smooth 
on a Zariski open set. 
We can use the theory of harmonic integral 
by taking a suitable complete K\"ahler from on this Zariski open set
(see \cite{Fuj12-A}). 
Then we represent a given cohomology class  
by the associated harmonic form  
$u_{\e}$ with respect to $h_{\e}$.

In Step 2, by taking a suitable (holomorphic) section $s$ 
of some positive multiple $F^{m}$ of $F$, 
we show that the norm of  
$D^{''*} su_{\e}$ converges to zero 
when $\e$ goes to zero, 
where $D^{''*} $ is 
the adjoint operator of $\dbar$. 
If the curvature  of $h_{\e}$ is semi-positive, 
we can see that $D^{''*} su_{\e}$  is zero. 
However it is unfortunately not semi-positive 
although the curvature of $h$ is semi-positive. 
For this reason, we need to generalize Enoki's technique for 
Koll\'ar's injectivity theorem 
by investigating the asymptotic behavior of $u_{\e}$.

In Step 3, we solve the $\dbar$-equation with $L^2$-estimates. 
By considering $su_{\e}$ instead of $u_{\e}$, 
we construct a solution $\beta_{\e}$ of 
the $\dbar$-equation $\dbar \beta_{\e} = su_{\e}$. 
Moreover, in this step, we show that 
the $L^{2}$-norm $\| \beta_{\e} \|$  
is uniformly bounded 
from the special characteristics of 
$h_{\min}$.

In Step 4, we investigate the limit of $u_{\e}$. 
The above arguments imply 
\begin{equation*}
\| su_{\e} \| ^{2} = 
\lla su_{\e}, \dbar \beta_{\e} \rra
\leq \| D^{''*} su_{\e}\|
\| \beta_{\e} \| \to 0 \quad {\text{as }} \e \to 0.
\end{equation*}
From this convergence we prove that 
$u_{\e}$ converges to zero in a suitable sense. This completes the proof.

This paper is divided into four sections. 
In Section 2, we 
recall the fundamental results that
are often used in this paper.
In Section 3, we give a proof of the main result and its corollaries. 
In Section 4, we give examples that tell us 
that metrics with minimal singularities 
do not always have algebraic singularities.

This paper is a slightly revised version of the paper in arXiv:1306.2497v1. 
Ten months after we finished writing it, 
the strong openness conjecture was solved by Guan and Zhou in \cite{GZ13}, 
and another proof was given by Hiep in \cite{Hie14}. 
Although their celebrated result implies the main theorem, 
we believe that it is worth displaying our methods to handle some degenerate metrics. 
Our viewpoint is quite different from them 
and our techniques seem to have some applications (for example see \cite{Mat14B}). 
In fact, we can obtain analytic versions of the injectivity theorem 
by applying our techniques in \cite{Mat13-B} and applications to 
the problem of extending holomorphic sections in \cite{GM13}.

\subsection*{Acknowledgment}
The author wishes to express his deep gratitude to Professor Shigeharu Takayama 
who gave the main problem of this paper when he was a master course student. 
He also would like to thank to the referee for carefully reading 
our manuscript and for giving useful comments. 
He is supported by the Grant-in-Aid for 
Young Scientists (B) $\sharp$25800051 from JSPS.

\section{Preliminaries}

In this section, 
we recall fundamental results 
related to singular metrics and the theory of harmonic integrals. 
For more details, refer to \cite{Dem}, \cite{Dem-book}, \cite{Laz}.

\subsection{Singular Metrics and the Nadel Vanishing Theorem}
Throughout this subsection, 
let $X$ be a compact complex manifold and 
$F$ be a line bundle on $X$.    
First we recall the definition and properties of 
singular metrics and their multiplier ideal sheaves. 

\begin{defi}[Singular metrics, curvatures, algebraic singularities] 
\ \label{s-met}\vspace{0.2cm} \\
(1) A (hermitian) metric $h$ on $F$ is called 
a {\textit{singular metric}}, if  
for a local trivialization $\theta:F|_{U} \cong 
U \times \mathbb{C}$ and a local section $\xi$ of $F$ on 
an open set $U \subset X$, there exists an 
$L^{1}_{\rm{loc}}$-function $\varphi$ on $U$ such that 
\begin{equation*}
|\xi|_{h} = |\theta (\xi)| e^{- \varphi}. 
\end{equation*}
Here $\varphi$ is called the local \textit{weight} of $h$ 
with respect to the trivialization. 
\vspace{0.2cm} \\ 
(2) The {\textit{curvature current}} 
$\sqrt{-1} \Theta_{h}(F)$ 
associated to $h$ is defined by 
\begin{equation*}
\sqrt{-1} \Theta_{h}(F) = \ddbar \varphi, 
\end{equation*}
where $ \varphi$ is a local weight of $h$. 
\vspace{0.2cm} \\ 
(3) 
A singular metric $h$ on $F$ is said to have 
{\textit{algebraic}} (resp. {\textit{analytic}}) {\textit{singularities}}, 
if there exists an ideal sheaf 
$\mathcal{I} \subset \mathcal{O}_{X}$ such that 
a local weight $ \varphi$ of $h$ can be locally 
written as 
\begin{equation*}
\varphi = \frac{c}{2} \log \big( 
|f_{1}|^{2} + |f_{2}|^{2} + \cdots + |f_{k}|^{2}\big) +v, 
\end{equation*}
where $c \in \mathbb{Q}_{>0}$ 
(resp. $c \in \mathbb{R}_{>0}$), 
$f_{1}, \dots f_{k}$ are local generators of $\mathcal{I}$ 
and $v$ is a smooth function.

\end{defi}
In this paper, we simply abbreviate 
the singular metric (resp. the curvature current) 
to the metric (resp. the curvature). 
The Levi form $\ddbar \varphi$ in the definition of the curvature 
is taken in the sense of 
distributions.  
The curvature of singular metrics 
is not always smooth a $(1,1)$-form but $(1,1)$-current. 
The curvature $\sqrt{-1} \Theta_{h}(F)$ 
is said to be {\textit{semi-positive}} (resp. {\textit{positive}}) 
if $\sqrt{-1} \Theta_{h}(F) \geq 0$
(resp. $\sqrt{-1} \Theta_{h}(F) \geq \omega$ for some hermitian form $\omega$) 
in the sense of $(1,1)$-currents. 
The curvature $\sqrt{-1} \Theta_{h}(F)$ is semi-positive 
if and only if a local weight $ \varphi$ of $h$ is a psh function. 
For a given singular metric $h$, 
we can define the multiplier ideal sheaf $\I{h}$. 
It is a coherent ideal sheaf that measures singularities of $h$ 
by using $L^2$-integrability of holomorphic functions.

\begin{defi}
Let $h$ be a singular metric on $F$ such that 
$\sqrt{-1} \Theta_{h}(F) \geq \gamma$  
for some smooth $(1,1)$-form $\gamma$ on $X$. 
Then the ideal sheaf $\I{h}$ 
defined to be 
\begin{equation*}
\I{h}(U):= \{f \in \mathcal{O}_{X}(U)\mid 
|f|e^{-\varphi} \in L^{2}_{\rm{loc}}(U) \}
\end{equation*}
for an open set $U \subset X$ is called the \textit{multiplier ideal sheaf} associated to $h$. 
\end{defi}

The original Nadel vanishing theorem (Theorem \ref{Nad}) 
asserts that all higher cohomology groups 
with coefficients in $K_{X} \otimes F \otimes \I{h}$
vanish if $h$ has \lq \lq strictly" positive curvature. 
Our motivation is to obtain a Nadel vanishing theorem 
for special singular metrics with semi-positive curvature, 
in particular metrics with minimal singularities $h_{\mi}$. 
The metric $h_{\mi}$ never has strictly positive 
curvature except the trivial case where $F$ is an ample line bundle (for example see \cite{Mat14A}). 
Nevertheless, from the reason stated in Section 1, 
we can obtain a Nadel vanishing theorem    
for $h_{\mi}$ by the special property of minimal singularities. 
Let us recall the definition of metrics with minimal singularities. 
\begin{defi}
Let $h_{1}$ and $h_{2}$ be singular metrics on 
$F$ with semi-positive curvature. 
The metric $h_{1}$ is said to 
be {\textit{less singular}} than $h_{2}$ 
if $h_{1} \leq C h_{2}$ for some positive constant $C>0$. 
\end{defi}
When $F$ is pseudo-effective 
(that is, $F$ admits a singular metric with semi-positive curvature), 
we can construct a metric $h_{\min}$ on $F$ with the following properties: 
\begin{itemize}
\item[$\bullet$] $h_{\min}$ has semi-positive curvature.
\item[$\bullet$] $h_{\min}$ is less singular than any metric on $F$ with semi-positive curvature.
\end{itemize}
Such a metric is uniquely determined up to 
equivalence of singularities  
(see \cite[(6.4) Theorem]{Dem} for more details). 
Metrics with minimal singularities  
do not always have analytic singularities. 
Indeed, if $h_{\mi}$ has analytic singularities, 
then $F$ admits a birational Zariski decomposition 
(see Section 4). 
However, there exist big line bundles 
that do not admit a birational Zariski decomposition. 

\subsection{Equisingular Approximations}
In the proof of the main result, 
we apply the equisingular approximation to a given singular metric. 
In this subsection, 
we reformulate \cite[Theorem 2.3.]{DPS01} with our notation 
and give further remarks. 

\begin{theo}$($\cite[Theorem 2.3.]{DPS01}$).$ \label{equi}
Let $X$ be a compact K\"ahler manifold and $F$ be a line bundle 
with a metric $h$ 
with semi-positive curvature. 
Then there exist  
singular metrics $\{h_{\e} \}_{1\gg \e>0}$ on $F$ with  the 
following properties:
\begin{itemize}
\item[(a)] $h_{\e}$ is smooth on $X \setminus Z_{\e}$, where $Z_{\e}$ is a subvariety on $X$.
\item[(b)] $h_{\e_{2}} \leq h_{\e_{1}} \leq h$ holds for any $0< \e_{1} < \e_{2} $. 
\item[(c)] $\I{h}= \I{h_{\e}}$.
\item[(d)] $\sqrt{-1} \Theta_{h_{\e}}(F) \geq -\e \omega$.
\end{itemize}
Moreover, if the set $\{x \in X \mid \nu(\varphi, x) >0 \}$ is contained in 
a subvariety $Z$, then we can add the property 
that $Z_{\e} $ is contained in $Z$ for any $\e > 0$.  
Here $\nu(\varphi, x)$ denotes the Lelong number at $x$ 
of a local weight $\varphi$ of $h$. 
\end{theo}
\begin{proof}
Fix a smooth metric $g$ on $F$. 
Then there exists an $L^{1}$-function $\varphi$ on $X$ 
with $h=g e^{- \varphi}$. 
By applying \cite[Theorem 2.3.]{DPS01} to $\varphi$, 
we can obtain quasi-psh functions 
$\varphi _{\nu}$ with equisingularities. 
For a given $\e>0$, by taking a large $\nu=\nu(\e)$,
we define $h_{\e}$ by $h_{\e}:=g e^{- \varphi_{\nu(\e)}}$. 
Then the metric $h_{\e}$ satisfies properties (a), (b), (c), (d). 

The latter conclusion follows from the proof. 
We will see this fact shortly by using the notation in \cite{DPS01}. 
In their proof, they locally approximate $\varphi$ by 
$\varphi_{\e, \nu, j}$ with logarithmic pole. 
By inequality (2.5) in \cite{DPS01}, 
the Lelong number of 
$\varphi_{\e, \nu, j}$ is 
less than or equal to that of $\varphi$. 
Hence   
$\varphi_{\e, \nu, j}$ is smooth on $X \setminus Z$  
since $\varphi_{\e, \nu, j}$ has a logarithmic pole. 
Then $\varphi _{\nu}$ is obtained from 
Richberg's regularization of 
the supremum of 
these functions (see around (2.5) and (2.7)). 
Since  the supremum is continuous on $X \setminus Z$,  
we obtain the latter conclusion. 
\end{proof}

\subsection{The Theory of Harmonic Integrals}
We recall some facts on the theory of harmonic integrals. 
Throughout this subsection, 
let $Y$ be a (not necessarily compact) complex manifold with a hermitian form $\omega$ 
and $E$ be a line bundle  on $Y$ with a smooth metric $h$.

For $E$-valued $(p,q)$-forms  $u$ and $v$,  
the pointwise inner product 
$\langle u, v\rangle _{h, \omega}$
can be defined, and the (global) inner product 
$\lla u, v \rra  _{h, \omega}$ 
can also be defined by
\begin{equation*}
\lla u, v \rra  _{h, \omega}:=
\int_{Y} 
\langle u, v\rangle _{h, \omega}\ \frac{\omega^{n}}{n!}.
\end{equation*}
Then the $L^{2}$-space of $E$-valued  $(p, q)$-forms 
is defined as follows:  
\begin{equation*}
L_{(2)}^{p, q}(Y, E)_{h, \omega}:= 
\{u \mid u \text{ is an }E \text{-valued\ } 
(p, q)\text{-form with } 
\|u \|_{h, \omega}< \infty \}. 
\end{equation*}
The connection $D_{h}$ on $E$  
is determined by the holomorphic structure of $E$ and 
the hermitian metric $h$, 
which is called the Chern connection. 
The Chern connection $D_{h}$ can be written as 
$D_{h} = D'_{h} + D''_{h}$ with the $(1,0)$-connection $D'_{h}$ and the 
$(0,1)$-connection $D''_{h}$. 
The $(0,1)$-connection $D''_{h}$ agrees with $\dbar$ 
by the definition.
The connections $D'_{h}$ and $D''_{h}$ 
can be seen as a closed and densely defined operator on 
$L_{(2)}^{p, q}(Y, E)_{h, \omega}$. 
Remark that the formal adjoints $D^{'*}_{h}$ and $D^{''*}_{h}$ 
coincide with the Hilbert space adjoints 
if $\omega$ is a \textit{complete} metric on $Y$ 
(see \cite[(3,2) Theorem in Chapter 8]{Dem-book}).

\begin{prop}\label{Nak}
Let $Y$ be a complex manifold of dimension $n$ and 
$ \omega$ be a complete K\"ahler metric on $Y$. 
If the inequality $\sqrt{-1}\Theta_{h}(E) \geq  - C \omega$ 
holds for some positive constant $C>0$, 
then we have the following equality 
\begin{equation*}
\| D_{h}^{''*}u \|_{h, \omega}^{2} + 
\|\overline{\partial} u \|_{h, \omega}^{2} = 
\| D_{h}^{'*}u  \|_{h, \omega}^{2} + 
\lla \sqrt{-1}\Theta_{h}(E)\Lambda_{\omega} u, u
\rra_{h, \omega} 
\end{equation*} 
for every $u \in L_{(2)}^{n, i}(Y, E)_{h, \omega}$ 
with $u \in {\rm{Dom}}D_{h}^{''*} \cap {\rm{Dom}} 
\overline{\partial}$.   
Here $\Lambda_{\omega}$ 
denotes the adjoint operator of the wedge product $\omega \wedge \cdot$. 
\end{prop}
\begin{proof}
This proposition is obtained from   
the Bochner-Kodaira-Nakano identity:  identity and the density lemma (cf. \cite{DPS01}). 
Since $\omega$ is a K\"ahler form, 
we have the Bochner-Kodaira-Nakano identity: 
\begin{equation*}
\Delta^{''} = \Delta^{'} + [\sqrt{-1}\Theta_{h}(F), \Lambda_{\ome}].  
\end{equation*}
Here $\Delta^{'}$ (resp. $\Delta^{''}$) 
is the Laplacian operator defined by 
$\Delta^{'}:= D_{h}^{'}D_{h}^{'*} +D_{h}^{'*} D_{h}^{'}$ (resp. 
$\Delta^{''}:= D_{h}^{''}D_{h}^{''*} +D_{h}^{''*} D_{h}^{''}$)  
and [$\cdot$\ , $\cdot$] is the graded Lie bracket. 
The Bochner-Kodaira-Nakano identity implies that 
the equality in the proposition holds 
if $u$ is smooth and compactly supported.

Since $\omega$ is complete, we can take a family of 
cut-off functions $\{ \psi_{\ell}\}_{\ell=1}^{\infty}$ 
with $|d\psi_{\ell}|_{\ome}\leq 1$. 
For a given $u$ satisfying $u \in {\rm{Dom}}D_{h}^{''*} \cap {\rm{Dom}} 
\overline{\partial}$, 
by putting $u_{\ell}:= u \psi_{\ell}$ and 
considering convolution with regularizing kernels $\rho_{\e}$, 
we obtain 
$u_{\ell,\e}:=u_{\ell}*\rho_{\e}$ satisfying 
the following properties: 
\begin{enumerate}
\item[$\bullet$] $u_{\ell,\e}$ is smooth and compactly supported. 
\item[$\bullet$] $u_{\ell,\e}$ converges to $u$ in $L^{n,q}_{(2)}(Y, E)_{h, \omega}$. 
\item[$\bullet$] 
$D_{h}^{''*} u_{\ell,\e}$ (resp. $\overline{\partial}u_{\ell,\e}$) 
converges to $D_{h}^{''*}u$ (resp. $\overline{\partial}u$) in 
$L^{n,\bullet}_{(2)}(Y, E)_{h, \omega}$. 
\end{enumerate}
The third property comes from $|d\psi_{\ell}|_{\ome}\leq 1$
(completeness of $\ome$). 
When $\e$ goes to zero, 
we have the equality for $u_{\ell}$ in the proposition. 
Further 
by the assumption of $\sqrt{-1}\Theta_{h}(E) \geq  - C \omega$, 
the integrand of the second term 
\begin{equation*}
\lla \sqrt{-1}\Theta_{h}(E)\Lambda_{\omega} u_{\ell}, u_{\ell} 
\rra_{h, \omega}
= \int_{Y} \langle \sqrt{-1}\Theta_{h}(E)\Lambda_{\omega} u_{\ell}, u_{\ell}
\rangle_{h, \omega}\ \frac{\omega^{n}}{n!}.
\end{equation*}
is semi-positive if we add the term $C|u_{\ell}|^{2}_{h, \omega}$. 
Therefore 
we obtain the conclusion 
by Lebesgue's monotone convergence theorem.  
\end{proof}

\subsection{On the Openness Conjecture}
It is worth mentioning that 
the main result of this paper is related 
to the (strong) openness conjecture on singularities of psh functions 
(see \cite[5.3 Remark]{DK01}).

When $h_{\min}$ is a metric with minimal singularities on a big line bundle $F$, 
we have  
\begin{equation*}
\mathcal{I}_{+}(h_{\mi}) \subset 
\mathcal{I}(\| F \|) \subset 
\mathcal{I}(h_{\mi}), 
\end{equation*}
where $\mathcal{I}_{+}(h_{\mi})$ is 
the {\textit{upper semi-continuous regularization}}  
of the multiplier ideal sheaf of $h_{\min}$ 
(see \cite{DEL00} for the definition). 
Further we can easily see that 
all higher cohomology groups with coefficients in 
$K_{X}\otimes F \otimes \mathcal{I}_{+}(h_{\mi})$ vanish.  
Indeed, it is easy to check that there exists a metric $h$ on $F$  
such that the curvature of $h$ is strictly positive and $\I{h}=\mathcal{I}_{+}(h_{\mi})$. 
Therefore we can directly obtain a vanishing theorem 
for $\mathcal{I}_{+}(h_{\mi})$ 
by the original Nadel vanishing. 
We can easily see that 
the multiplier ideal sheaf $\I{\varphi}$ agrees with 
$\mathcal{I}_{+}(\varphi)$  
when $\varphi$ has analytic singularities,  
but unfortunately $h_{\min}$ does not always have 
algebraic singularities (see Section 4).  
Further it is difficult to handle multiplier ideal sheaves  
of singular metric with non-algebraic 
(transcendental) singularities. 
For this reason, the main problem of this paper was 
an open problem even if $F$ is big.

It is an interesting question to ask whether or not 
$\mathcal{I}_{+}(\varphi)$ agrees with $\mathcal{I}(\varphi)$ 
for a psh function $\varphi$,  
which is called the strong openness conjecture (see \cite{DEL00}). 
It is a natural question, but it had been an open problem 
before Guan and Zhou solved the strong openness conjecture in \cite{GZ13} 
(see also \cite{Hie14} and \cite{Ber13}). 
Although their celebrated results imply the main theorem, 
our proof does not need the openness conjecture.


\section{Proof of the Main Results}
\subsection{Proof of Theorem \ref{main}}
In this subsection, 
we prove the following theorem, 
which leads to Theorem \ref{m-min} and Corollary \ref{main-co}. 

\begin{theo}\label{main}
Let $F$ be a big line bundle on a smooth projective 
variety $X$ and $h$ be 
a singular metric on $F$ 
with semi-positive curvature. 
Assume that
there exist a $($non-zero$)$  section $s$ of 
some positive multiple $F^{m}$ of $F$ and 
a singular metric $h_{\rbig}$ on $F$ satisfying 
the following conditions$:$ \\
\hspace{0.4cm}$(1)$ 
The metric $h_{\rbig}$ has strictly positive curvature.\\  
\hspace{0.4cm}$(2)$ 
The metric $h$ is less singular than $h_{\rbig}$.\\
\hspace{0.4cm}$(3)$ 
The pointwise norm $|s|_{h^m_{\rbig}}$ 
of $s$ with respect to $h^{m}_{\rbig}$ is bounded on $X$. 
\vspace{0.1cm}\\
Then we have 
\begin{equation*}
H^{i}(X, K_{X} \otimes F \otimes \I{h_{}}) = 0 
\quad {\text{for}}\ {\text{any}}\ i >0.
\end{equation*}
\end{theo}
\begin{rem}
If $h$ is a metric with minimal singularities or Siu's metric 
on a big line bundle $F$, we can construct $h_{\rbig}$ 
satisfying all conditions in the above theorem 
(see the proof of Corollary \ref{main-co}). 
\end{rem}
\textit{Proof of Theorem \ref{main}.}\ \ 
First of all, we will explain the idea of the proof.  
For simplicity, we first observe the case where $h$ is smooth on $X$ 
(or some Zariski open set of $X$). 
Take an arbitrary cohomology class
$\{u \} \in H^{i}(X, K_{X} \otimes F \otimes \I{h})$ 
represented by the $F$-valued 
$(n, i)$-form $u$ such that 
the norm $\|u \|_{h}$ is bounded and $u$ is harmonic.
If the norm $\| u \|_{h_{\rbig}}$ is bounded, 
we can construct a solution $\gamma \in 
L_{(2)}^{n, i-1}(X, F)_{h}$ of 
the $\dbar$-equation $ \dbar \gamma = u$ 
by assumptions (1), (2). 
This completes the proof. 
However we can never expect that 
$\|u \|_{\rbig}$ is bounded 
since $h_{\rbig}$ is more singular than $h$. 
Then we consider the $\dbar$-equation $\dbar \beta = su$ 
instead of $ \dbar \gamma = u$. 
Then the norm $\|su\|_{h_{\rbig}^{m}h}$ 
is bounded since 
the singularities of $h_{\rbig}$ are canceled by the zero of $s$ 
thanks to assumption (3).  
Further the metric $h_{\rbig}^{m}h$ has strictly positive curvature.
Therefore we can construct a solution 
$\beta \in L_{(2)}^{n, i-1}(X, F^{m+1})_{h^{m+1}}$ of 
the $\dbar$-equation $\dbar \beta = su$. 
On the other hand, $su$ can be shown to be harmonic 
since $u$ is harmonic and 
the curvature of $h$ is semi-positive. 
It follows from Enoki's technique 
for the injectivity theorem. 
From these arguments, we can conclude that $u$ is zero.

Unfortunately 
$h$ may not smooth even on Zariski open sets of $X$.
For this reason, in Step 1, 
we approximate a given singular metric $h$ by 
singular metrics $\{ h_{\e} \}_{\e>0}$
that are smooth on a Zariski open set. 
Then we need to investigate the asymptotic behavior 
of the harmonic forms that 
represent a given cohomology class 
since the positivity of the curvature of $h_{\e}$ 
can be lost. 
In this step, we  
fix the notation to apply the theory of harmonic integrals.  
At the end of this step, we give the sketch of the proof. 
\vspace{0.1cm}\\
{\bf{Step 1 (The equisingular approximation of $h$)}}
\vspace{0.1cm}\\
We can choose  a K\"ahler form $\omega$ on $X$ such that  
$\sqrt{-1} \Theta_{h_{\rbig}}(F) \geq \omega$ holds 
since the curvature of $h_{\rbig}$ is strictly positive. 
For the proof  
we apply the theory of harmonic integrals, 
but $h$ may not be smooth. 
For this reason, we need to approximate $h$ by metrics 
$\{ h_{\e} \}_{\e>0}$ that are smooth on a Zariski open set. 
By Theorem \ref{equi}, we can obtain 
metrics $\{ h_{\e} \}_{\e>0}$ on $F$ with the 
following properties:
\begin{itemize}
\item[(a)] $h_{\e}$ is smooth on $X \setminus Z_{\e}$, where $Z_{\e}$ is a subvariety on $X$.
\item[(b)] $h_{\e_{2}} \leq h_{\e_{1}} \leq h$ holds for any 
$0< \e_{1} < \e_{2} $.
\item[(c)] $\I{h}= \I{h_{\e}}$.
\item[(d)] $\sqrt{-1} \Theta_{h_{\e}}(F) \geq -\e \omega$.
\end{itemize}
By assumptions (2), (3), 
the pointwise norm $|s|_{h^{m}}$ is also bounded on $X$,  
and thus the set $\{x \in X \mid \nu(\varphi, x)>0\}$ 
is contained in the subvariety 
$Z:=\{z \in X \mid s(x)=0 \}$,  
where $\nu(\varphi, x)$ denotes the Lelong number  at $x \in X$ 
of a local weight $\varphi$ of $h$. 
From the latter conclusion of Theorem \ref{equi} 
we know that $h_{\e}$ is smooth on $X \setminus Z$. 
Therefore   
we can add a stronger property than property $(a)$, namely
\begin{itemize}
\item[(e)] $h_{\e}$ is smooth on $Y:=X \setminus Z$, 
where Z is a subvariety independent of $\e$. 
\end{itemize}

Now we construct a complete K\"ahler metric on $Y$ 
with suitable potential function. 
Take a quasi-psh function $\psi$ on $X$ such that  
$\psi$ is smooth on $Y$ and 
$\psi$ has a logarithmic pole along $Z$.
Since quasi-psh functions are upper semi-continuous, 
we may assume $\psi \leq -e$.
Then we define the function $\varphi$ and 
the $(1,1)$-form $\ome$ on $Y$ by 
\begin{equation*}
\varphi:=\frac{1}{\log(-\psi)}\quad \text{and}\quad  
\ome:= \ell \omega + \ddbar \varphi, 
\end{equation*}
where $\ell$ is a positive number. 
By taking a sufficiently large $\ell >0$, 
we can easily see that   
the $(1,1)$-form $\ome$ satisfies the following properties:  
\begin{itemize}
\item[(A)] $\ome$ is a complete K\"ahler metric on $Y$.
\item[(B)] $\varphi$ is bounded on $X$.
\item[(C)] $\ome \geq \omega $.
\end{itemize}
In fact, properties (B), (C) are obvious. 
Property (A) follows from straightforward computations 
(see \cite[Lemma 3.1]{Fuj12-A}).

In the proof, we consider 
the harmonic forms on $Y$ with respect to 
$h_{\e}$ and $\ome$ (not $h$ and $\omega$). 
Let $L_{(2)}^{n, i}(Y, F)_{h_{\e}, \ome}$ be 
the space of square integrable $F$-valued $(n,i)$-forms 
$\alpha$ with respect to the inner product defined by 
\begin{equation*}
\|\alpha \|^{2}_{h_\e, \ome}:= \int_{Y} 
|\alpha |^{2}_{h_{\e}, \ome}\ \frac{\ome^{n}}{n!}. 
\end{equation*}
Then we have the orthogonal decomposition 
\begin{equation*}
L_{(2)}^{n, i}(Y, F)_{h_{\e}, \ome}
=
{\rm{Im}} \dbar
\oplus
\mathcal{H}^{n, i}(Y, F)_{h_{\e}, \ome}
\oplus {\rm{Im}} D^{''*}_{h_{\e}}.  
\end{equation*}
Here the operators $D^{'*}_{h_{\e}}$
$D^{''*}_{h_{\e}}$ are
the Hilbert space adjoints of 
$D^{'}_{h_{\e}}$ and $D^{''}_{h_{\e}}=\dbar$. 
(Since $\ome$ is complete, these coincide with 
the closed extensions of the formal adjoints in the sense of distributions.) 
The space $\mathcal{H}^{n, i}(Y, F)_{h_{\e}, \ome}$ denotes  
the space of harmonic forms with respect to 
$h_{\e}$ and $\ome$, 
namely 
\begin{equation*}
\mathcal{H}^{n, i}(Y, F)_{h_{\e}, \ome}= 
\{\alpha   \mid \alpha 
\text{ is an } F\text{-valued } (n,i)\text{-form such that  }
\dbar \alpha= D^{''*}_{h_{\e}}\alpha=0    \}. 
\end{equation*}
Harmonic forms are smooth by the regularization theorem 
for elliptic operators. 
These facts may be known to specialists.  
The precise proof can be found in 
\cite[Claim 1]{Fuj12-A}.

From property (C) 
we have 
the inequality 
$|\beta|^{2}_{\ome}\ \ome^{n} \leq 
|\beta|^{2}_{\omega}\ \omega^{n}$  for any $(n, i)$-form $\beta$. 
From this inequality and property (b) of $h_{\e}$,
we obtain 
\begin{equation}\label{ine}
\|\alpha \|_{h_{\e}, \ome} \leq 
\|\alpha \|_{h_{\e}, \omega} \leq 
\|\alpha \|_{h, \omega}
\end{equation}
for any $F$-valued $(n,i)$-form $\alpha$. 
This inequality is often used in the proof.

Take an arbitrary cohomology class
$\{u \} \in H^{i}(X, K_{X} \otimes F \otimes \I{h})$ 
represented by an $F$-valued 
$(n, i)$-form $u$ with $\|u \|_{h, \omega} < \infty$. 
Our goal is to show that 
the cohomology class 
$\{u \}$ is actually zero.  
By inequality (\ref{ine}), we know $u \in L_{(2)}^{n, i}(Y, F)_{h_{\e}, \ome}$ for any $\e > 0$. 
By the orthogonal decomposition above, 
there exist  
$u_{\e} \in \mathcal{H}^{n, i}(Y,F)_{h_{\e}, \ome}$ and 
$v_{\e} \in L_{(2)}^{n,i-1}(Y, F)_{h_{\e}, \ome}$ such that 
\begin{equation*}
u=u_{\e}+\dbar v_{\e}. 
\end{equation*}
Notice that  
the component of ${\rm{Im}} D^{''*}_{h_{\e}}$ is zero 
since $u$ is $\dbar$-closed.

The strategy of the proof is as follows: 
In Step 2, we show 
\begin{equation*}
\|D^{''*}_{h^{m+1}_{\e}} s u_{\e} \|_{h^{m+1}_{\e}, \ome} 
\to 0  \quad {\text{as }} \e \to 0.
\end{equation*} 
This step can be regarded as a generalization of  
Enoki's proof for Koll\'ar's injectivity theorem. 
In Step 3, 
we construct a solution $\beta_{\e}$ of the $\dbar$-equation 
$\dbar \beta_{\e} = s u_{\e}$ such that 
the norm $\| \beta_{\e} \|_{h^{m+1}_{\e}, \ome}$ 
is uniformly bounded. 
By these steps, we know    
\begin{equation*}
\| su_{\e} \|_{h^{m+1}_{\e}, \ome} ^{2} = 
\lla su_{\e}, \dbar \beta_{\e} \rra_{h^{m+1}_{\e}, \ome}
\leq \| D^{''*}_{h^{m+1}_{\e}} su_{\e}\|_{h^{m+1}_{\e}, \ome}
\| \beta_{\e} \|_{h^{m+1}_{\e}, \ome} \to 0 \quad {\text{as }} \e \to 0.
\end{equation*}
In Step 4, we show that 
$u_{\e}$ converges to zero in a suitable sense. 
This completes the proof. 
\vspace{0.1cm}\\
{\bf{Step 2 (A generalization of Enoki's proof 
of the injectivity theorem)}}\\
The aim of this step is to prove the following claim. 
\vspace{-0.1cm}
\begin{claim} \label{D''}
The norm 
$\|D^{''*}_{h^{m+1}_{\e}} s u_{\e} \|_{h^{m+1}_{\e}, \ome}$ converges to zero as letting $\e$ go to zero. 
\end{claim}
\hspace{-0.6cm}
\textit{Proof of Claim \ref{D''}.}\ \ 
From the definition of  $u_{\e}$ and 
inequality (\ref{ine}), we have 
\begin{equation}\label{ine2}
\|u_{\e} \|_{h_{\e}, \ome} 
\leq \|u \|_{h_{\e}, \ome} 
\leq \|u \|_{h, \omega}.
\end{equation}
In the proof, these inequalities play an important role. 
By applying Proposition \ref{Nak} to $u_{\e}$, 
we obtain 
\begin{equation}\label{B-eq}
0 = \lla \sqrt{-1}\Theta_{h_{\e}}(F)
\Lambda_{\ome} u_{\e}, u_{\e}
  \rra_{h_{\e}, \ome} +
\|D^{'*}_{h_{\e}}u_{\e} \|^{2}_{h_{\e}, \ome}. 
\end{equation}
Note that the left hand side is zero since $u_{\e}$ is harmonic.
Let $A_{\e}$ be the first term and $B_{\e}$ 
the second term of the right hand side of equality (\ref{B-eq}). 
From now on, 
we will show that the first term $A_{\e}$ and 
the second term $B_{\e}$ converge to zero. 
For simplicity, 
we put  
\begin{equation*}
g_{\e}:= \langle  \sqrt{-1}\Theta_{h_{\e}}(F)
\Lambda_{\ome} u_{\e}, u_{\e}
 \rangle_{h_{\e}, \ome}. 
\end{equation*}
Then we can easily see that  
there exists a positive constant $C$ such that
\begin{equation}\label{ine3}
g_{\e} \geq -\e C |u_{\e}|^{2}_{h_{\e}, \ome}. 
\end{equation}
Indeed, let $\lambda_{1}^{\e} \leq \lambda_{2}^{\e} \leq 
\dots \leq \lambda_{n}^{\e} $ be the 
eigenvalues of $\sqrt{-1}\Theta_{h_{\e}}(F)$ with respect to 
$\ome$. 
Then for any point $y \in Y$, there exists 
a local coordinates $(z_{1}, z_{2}, \dots, z_{n})$ 
centered at $y$ such that 
\begin{align*}
\sqrt{-1}\Theta_{h_{\e}}(F) = \sum_{j=1}^{n} 
\lambda_{j}^{\e} dz_{j} \wedge d\overline{z_{j}}
\quad \text{and} \quad 
\ome = \sum_{j=1}^{n} 
 dz_{j} \wedge d\overline{z_{j}}
\quad {\rm{ at}}\ y. 
\end{align*}
When we locally write $u_{\e}$ as 
$u_{\e} =\sum_{|K|=i} f_{K}^{\e}\ dz_{1}\wedge \dots \wedge dz_{n} 
\wedge d\overline{z}_{K}$, 
we have  
\begin{equation*}
g_{\e}= \sum_{|K|=i} 
\Big{(} \sum_{j \in K} \lambda_{j}^{\e} \Big{)} 
|f_{K}^{\e}|^{2}_{h_{\e}}
\end{equation*}
by straightforward computations. 
On the other hand, from 
property (C) of $\ome$ and property (d) of $h_{\e}$, 
we have  
$\sqrt{-1}\Theta_{h_{\e}}(F) 
\geq -\e \omega 
\geq -\e \ome$, 
which implies that $\lambda_{j}^{\e} \geq -\e$.
Therefore we obtain inequality (\ref{ine3}).

From inequalities (\ref{ine2}), (\ref{ine3}) 
and equality (\ref{B-eq}), 
we obtain 
\begin{align*}
0 \geq A_{\e} &= \int_{Y} g_{\e}\ \ome^{n} \\
& \geq -\e C \int_{Y} |u_{\e}|^{2}_{h_{\e}, \ome}\ \ome^{n}\\
& \geq -\e C \|u \|^{2}_{h, \omega}. 
\end{align*}
Hence $A_{\e}$ converges to zero. 
By equality (\ref{B-eq}), $B_{\e}$ also converges to zero.

We will apply Proposition \ref{Nak} to $su_{\e}$. 
We first see that 
the norm $\|su_{\e} \|_{h^{m+1}_{\e}, \ome}$ is bounded. 
By assumptions (2), (3), 
the pointwise norm $|s|_{h^{m}}$ with respect to $h^{m}$ is 
bounded. 
Further, we have $|s|_{h_{\e}^{m}} \leq |s|_{h^{m}}$
from property (b) of $h_{\e}$,  
thus we obtain 
\begin{equation*}
\|s u_{\e} \|_{h_{\e}^{m+1}, \ome} \leq 
\sup_{X} |s|_{h_{\e}^{m}}  \|u_{\e} \|_{h_{\e}, \ome}  \leq
\sup_{X} |s|_{h^{m}}  \|u \|_{h, \omega} < \infty. 
\end{equation*}
By applying Proposition \ref{Nak} to $su_{\e}$, 
we obtain 
\begin{equation}\label{B-eq2}
\|D^{''*}_{h_{\e}^{m+1}}
su_{\e} \|^{2}_{h^{m+1}_{\e}, \ome} = \lla \sqrt{-1}\Theta_{h^{m+1}_{\e}}(F^{m+1})
\Lambda_{\ome} su_{\e}, su_{\e}
  \rra_{h^{m+1}_{\e}, \ome} +
\|D^{'*}_{h_{\e}^{m+1}}su_{\e} \|^{2}_{h^{m+1}_{\e}, \ome}.
\end{equation}
Here we used $\dbar s u_{\e}=0$. 
First we  
see that  
the second term of the right hand side converges to zero. 
Since $s$ is a holomorphic section, we can check   
the equality $D^{'*}_{h^{m+1}_{\e}}su_{\e} =
s D^{'*}_{h_{\e}}u_{\e}$, 
which  yields 
\begin{equation*}
\|D^{'*}_{h_{\e}^{m+1}}su_{\e} \|^{2}_{h^{m+1}_{\e}, \ome} \leq
\sup_{X}|s|^{2}_{h_{\e}^{m}}  \int_{Y}
|D^{'*}_{h_{\e}}u_{\e} |^{2}_{h_{\e}, \ome}\ \ome^{n} 
\leq \sup_{X}|s|^{2}_{h^{m}} B_{\e}. 
\end{equation*}
Since $B_{\e}$ converges to zero, 
the norm $\|D^{'*}_{h^{m+1}_{\e}}su_{\e} \|_{h^{m+1}_{\e}, \ome}$ also converges to zero.

For the proof of the claim, 
it remains to show that 
the first term of the right hand side of equation (\ref{B-eq2})
converges to zero. 
For this purpose, 
we investigate $A_{\e}$ in details. 
By the definition of $A_{\e}$, we have
\begin{equation*}
A_{\e}= \int_{\{ g_{\e} \geq 0 \}} g_{\e}\ \ome^{n} + 
\int_{\{ g_{\e} \leq 0 \}} g_{\e}\ \ome^{n}. 
\end{equation*}
Let $A^{+}_{\e}$ be the first term and 
$A^{-}_{\e}$ be the second term of the right hand side. 
Then  
$A^{+}_{\e}$ and $A^{-}_{\e}$ converge to zero. 
Indeed, by simple computations and 
inequalities (\ref{ine2}), (\ref{ine3}), 
we have 
\begin{align*}
0 \geq A^{-}_{\e} &\geq 
-\e C \int_{\{ g_{\e} \leq 0 \}}|u_{\e}|^{2}_{h_{\e}, \ome}\ 
\ome^{n} \\
&\geq -\e C \int_{Y}|u_{\e}|^{2}_{h_{\e}, \ome}\ 
\ome^{n} \\
&\geq -\e C \| u \|^{2} _{h, \omega}. 
\end{align*}
From this inequality, we know that $A^{+}_{\e}$ and $A^{-}_{\e}$ go to zero 
by $A_{\e} = A^{+}_{\e} + A^{-}_{\e}$. 
Combining
this convergence with the fact that $|s|_{h^m_{\e}}$
is uniformly bounded, the first term of the right
hand side of equality (\ref{B-eq2}) tends to zero. 
This completes the proof. 
\qed 
\vspace{0.1cm}\\
{\bf{Step 3 (Solutions of the $\dbar$-equation 
with uniformly bounded $L^{2}$-norms)}}
\vspace{0.1cm}\\ 
A positive multiple $s^{k}$ of the section $s$ also 
satisfies assumption (3),  
and thus we may assume that $m$ is greater than $\ell$.
Under this additional assumption, 
we prove 
the following claim.  

\begin{claim} \label{sol}
For every $\e>0$, there exists an  
$F^{m}$-valued $(n, i-1)$-form $\beta_{\e}$ 
such that 
\begin{equation*}
\dbar \beta_{\e} = su_{\e} \quad \text{and} \quad 
\ \| \beta_{\e} \|_{h^{m+1}_{\e}, \ome} \text{ is uniformly bounded.}
\end{equation*}
\end{claim}
\hspace{-0.6cm}
\textit{Proof of Claim \ref{sol}.}\ \ 
To construct solutions of the $\dbar$-equation 
$ \dbar \beta_{\e} = su_{\e} $, 
we apply \cite[Th\'eor\`em 4.1]{Dem82}. 
For this purpose, 
we consider the new metric $H_{\e}$ on $F^{m+1}$ defined by 
$H_{\e}:= h^{m}_{\rbig}h_{\e} e^{-\varphi}$.
First we compute the curvature of  
$H_{\e}$. 
We have  
\begin{align*}
\sqrt{-1}\Theta_{H_{\e}}(F^{m+1}) 
&= m \sqrt{-1} \Theta_{h_{\rbig}}(F)+ \sqrt{-1}\Theta_{h_{\e}}(F)+ \ddbar \varphi \\
&\geq 
m \omega - \e \omega + \ddbar \varphi. 
\end{align*}
By the definition of $\ome$ and 
the additional assumption of $m>\ell$, 
we obtain $\sqrt{-1}\Theta_{H_{\e}}(F^{m+1}) \geq \ome$. 
Now we consider the norm of $su_{\e}$ with respect to 
$H_{\e}$. 
By assumption (3), 
the pointwise norm $|s|_{h^{m}_{big}}$ is bounded on $X$.
Further, since $\varphi$ is bounded on $X$, there 
exist positive constants $C_{1}$ and $C_{2}$ 
such that $C_{1} \leq e^{-\varphi} \leq C_{2}$ on $X$. 
Then by inequality (\ref{ine2}) we obtain  
\begin{align*}
\|su_{\e} \|_{H_{\e}, \ome} 
&\leq C_{2} 
\sup_{X} |s|_{h^{m}_{\rbig}} \|u_{\e} \|_{h_{\e}, \ome} \\
&\leq C_{2}  \sup_{X} |s|_{h^{m}_{\rbig}} 
 \|u \|_{{h}, \omega}.
\end{align*} 
In particular, the norm $\|su_{\e} \|_{H_{\e}, \ome} $ 
is bounded. 
Moreover, the right hand side does not depend on $\e$. 
By applying \cite[Th\'eor\`em 4.1]{Dem82}, 
we can find a solution $\beta_{\e}$ 
of the $\dbar$-equation $ \dbar \beta_{\e} = su_{\e}$ 
with 
\begin{equation*}
\|\beta_{\e} \|^{2}_{H_{\e}, \ome} \leq 
\frac{1}{i}\|su_{\e} \|^{2}_{H_{\e}, \ome}.  
\end{equation*}
Since $h$ is less singular than $h_{\rbig}$, 
there exists a positive constant $C_{3} > 0$ such that 
$h_{\e} \leq h \leq C_{3} h_{\rbig}$. 
Then we can easily see 
\begin{equation*}
C^{-m}_{3} C_{1} 
\|\beta_{\e} \|_{h^{m+1}_{\e}, \ome} \leq
\|\beta_{\e} \|_{H_{\e}, \ome}. 
\end{equation*}
Since the right hand side 
can be estimated  by the constant independent of $\e$,   
the norm $\|\beta_{\e} \|_{h^{m+1}_{\e}, \ome}$ is uniformly bounded. 
This completes the proof. 
\qed 
\vspace{0.1cm}\\
{\bf{Step 4 (The limit of harmonic forms)}}
\vspace{0.1cm}\\ 
In this step, we investigate the limit of $u_{\e}$ and 
complete the proof of Theorem \ref{main}. 
First we prove the following claim. 
\begin{claim}\label{converge}
The norm $\| s u_{\e}\|_{h^{m+1}_{\e}, \ome}$ 
converges to zero as $\e \to \infty$. 
\end{claim}
\hspace{-0.6cm}
\textit{Proof of Claim \ref{converge}.}\ \ 
Take  
$\beta_{\e} \in L_{(2)}^{n, i-1}
(F^{m+1})_{h^{m+1}_{\e}, \ome}$ satisfying the properties 
in Claim \ref{sol}.  
Then straightforward computations yield 
\begin{align*}
\| s u_{\e}\|_{h^{m+1}_{\e}, \ome}^{2}
&=\lla  s u_{\e}, \dbar \beta_{\e}  
 \rra_{h^{m+1}_{\e}, \ome} \\
&=
 \lla  D^{''*}_{h^{m+1}_{\e}}s u_{\e},\beta_{\e}  
 \rra_{h^{m+1}_{\e}, \ome} \\
&\leq 
\|D^{''*}_{h^{m+1}_{\e}} s u_{\e}\| _{h^{m+1}_{\e}, \ome}
\|\beta_{\e} \|_{h^{m+1}_{\e}, \ome}. 
\end{align*}
The norm of    
$\beta_{\e}$ is  uniformly bounded by Claim \ref{sol}. 
On the other hand, the norm 
$\|D^{''*}_{h^{m+1}_{\e}} s u_{\e}\| _{h^{m+1}_{\e}, \ome}$ 
converges to zero by Claim \ref{D''}. 
Therefore the norm $\| s u_{\e}\|_{h_{\e}, \ome}$ also converges to zero. 
\qed

From now on, 
we fix a small positive number $\e_{0}>0$. 
Then for any positive number $\e$ with $0< \e < \e_{0}$, 
by property (b) of $h_{\e}$, we obtain
\begin{equation*}
\|u_{\e} \|_{h_{\e_{0}}, \ome} \leq \|u_{\e} \|_{h_{\e}, \ome} 
\leq  \|u \|_{h, \omega}. 
\end{equation*}
These inequalities say that the norms of $\{ u_{\e} \}_{\e >0}$ with respect to 
$h_{\e_{0}}$ are
uniformly bounded. 
Therefore 
there exists a subsequence of $\{ u_{\e} \}_{\e >0}$ 
that converges to 
$\alpha \in L_{(2)}^{n, i}(Y, F)_{h_{\e_{0}}, \ome}$
with respect to the weak $L^{2}$-topology. 
For simplicity, we denote this subsequence 
by the same notation $\{ u_{\e} \}_{\e >0}$. 
Then we prove the following claim.

\begin{claim}\label{zero}
The weak limit $\alpha$ of $\{ u_{\e} \}_{\e >0}$ 
in $\in L_{(2)}^{n, i}(Y, F)_{h_{\e_{0}}, \ome}$
is zero. 
\end{claim}
\hspace{-0.6cm}
\textit{Proof of Claim \ref{zero}.}\ \ 
For every positive number $\delta>0$, 
we define the subset $A_{\delta}$ of $Y$ by 
$A_{\delta}:= \{x \in Y \mid  |s|^{2}_{h^{m}_{\e_{0}}} > \delta  \}$. 
Since a weight of $h_{\e_{0}}$ is upper semi-continuous, 
$|s|^{2}_{h^{m}_{\e_{0}}}$ is lower semi-continuous. 
Hence $A_{\delta}$ is an open set of $Y$. 
We estimate the norm of $u_{\e}$ on $A_{\delta}$. 
By easy computations, we have 
\begin{align*}
\| s u_{\e}  \|^{2}_{h^{m+1}_{\e}, \ome} 
&\geq \| s u_{\e}  \|^{2}_{h^{m+1}_{\e_{0}}, \ome} \\
&\geq \int_{A_{\delta}} |s|^{2}_{ h^{m}_{\e_{0}} } 
|u_{\e}|^{2}_{h_{\e_{0}}, \ome}\ \ome^{n} \\
&\geq \delta  \int_{A_{\delta}} 
|u_{\e}|^{2}_{h_{\e_{0}}, \ome}\ \ome^{n} 
\geq 0
\end{align*}
for any $\delta>0$.

Note that $u_{\e} |_{A_{\delta}}$ 
converges to $\alpha |_{A_{\delta}}$ with respect to 
the weak $L^{2}$ topology in 
$ L_{(2)}^{n, i}(A_{\delta}, F)_{h_{\e_{0}}, \ome}$,  
where $u_{\e} |_{A_{\delta}}$ (resp. $\alpha |_{A_{\delta}}$) is  
the restriction of  $u_{\e}$ (resp. $\alpha$) to $A_{\delta}$. 
Indeed, 
for every 
$\gamma \in L_{(2)}^{n, i}(A_{\delta}, F)_{h_{\e_{0}}, \ome}$, 
the inner product  
$\lla u_{\e} |_{A_{\delta}}, \gamma
\rra_{A_{\delta}} 
= \lla u_{\e}, \widetilde{\gamma} 
\rra_{Y} $
converges to 
$\lla \alpha, \widetilde{\gamma}
\rra_{Y} 
= \lla \alpha, \gamma  
\rra_{A_{\delta}} $, 
where  $\widetilde{\gamma}$ is the zero extension of $\gamma$ to 
$Y$. 
Since $u_{\e} |_{A_{\delta}}$ 
converges to $\alpha |_{A_{\delta}}$, 
we obtain
\begin{equation*}
\|\alpha |_{A_{\delta}} \|_{h_{\e_{0}}, \omega} 
\leq 
\liminf_{\e \to 0}\|u_{\e} |_{A_{\delta}} \|_{h_{\e_{0}},  \omega}=0. 
\end{equation*}
This is because the norm of the weak limit can 
be estimated by the limit inferior of norms of a sequence. 
Therefore 
we know that $\alpha |_{A_{\delta}} = 0$ for any $\delta>0$. 
By the definition of $A_{\delta}$, the union of 
$\{A_{\delta} \}_{\delta >0}$ is equal to $Y=X \setminus Z$. 
Hence the weak limit $\alpha $ is zero on $Y$.
\qed

By using Claim \ref{zero}, we complete the proof of Theorem \ref{main}. 
By the definition of $u_{\e}$, 
we have 
\begin{equation*}
u = u_{\e} + \dbar v_{\e}. 
\end{equation*}
Claim \ref{zero} implies 
that $\dbar v_{\e}$ converges to $u$ with respect to 
the weak $L^{2}$-topology. 
Then we can easily see that 
$u$ is a $\dbar$-exact form 
(that is, $u \in {\rm{Im}} \dbar$). 
This is because
the subspace ${\rm{Im}} \dbar$ 
is closed in 
$L_{(2)}^{n, i}(Y, F)_{h_{\e_{0}}, \ome}$ 
with respect to the weak $L^{2}$-topology. 
Indeed, for every 
$\gamma_{1} + D^{''*}\gamma_{2} \in
\mathcal{H}^{n, i}(Y, F)_{h_{\e_{0}}, \ome}
\oplus {\rm{Im}} D^{''*}_{h_{\e_{0}}}$, we have 
$$\lla u, \gamma_{1} + D^{''*}\gamma_{2}  \rra =
\lim_{\e \to 0}\lla \dbar v_{\e}, \gamma_{1} + D^{''*}\gamma_{2} 
\rra =0.$$
Therefore we know $u \in {\rm{Im}} \dbar$.

In summary, we proved that $u$ is a $\dbar$-exact form in 
$L_{(2)}^{n, i}(Y, F)_{h_{\e_{0}}, \ome}$. 
In other words, we obtained a solution 
$f \in L^{n, i-1}_{(2)}(Y, F)_{h_{\e_{0}}, \ome}$ 
of the $\dbar$-equation $u= \dbar f$. 
From this fact, we can find a solution 
$g \in L^{n, i-1}_{(2)}(Y, F)_{{h}, \omega}$ of 
the $\dbar$-equation $u = \dbar g$. 
Note that $g$ may not be different from $f$. 
It says that the cohomology class $\alpha=\{u\}$ is the zero class.

Indeed, we consider the Dolbeault cohomology groups obtained 
from the closed operator $\dbar_{h, \omega}$ 
(resp. $\dbar_{h_{\e_{0}}, \ome}$) 
defined on the $L^{2}$-space $L^{n, \bullet}_{(2)}(X, F)_{h, \omega}$ 
(resp. $L^{n, \bullet}_{(2)}(Y, F)_{h_{\e_{0}}, \ome}$ ). 
Then we have the following commutative diagram: 
$$
\begin{CD}
{{\rm{Ker}}\ \dbar_{h, \omega}} /{{\rm{Im}}\, \dbar_{h, \omega}} 
@>j>>{{\rm{Ker}}\ \dbar_{h_{\e_{0}}, \ome}} /{{\rm{Im}}
\, \dbar_{h_{\e_{0}}, \ome}}\\ 
@V{\phi_{1}}V\cong V  @V{\phi_{2}}V\cong V
\\ \check{H}^{i}(X, K_{X}\otimes F \otimes \I{h}) @=  
\check{H}^{i}(X, K_{X}\otimes F \otimes \I{h_{\e_{0}}}).  
\end{CD}
$$
Here $j$ is the map induced by the natural map from 
$L^{n, \bullet}_{(2)}(X, F)_{h, \omega}$ to 
$L^{n, \bullet}_{(2)}(Y, F)_{h_{\e_{0}}, \ome}$, 
and $\phi_{i}$ is the De Rham-Weil isomorphism to 
the $\rm{\check{C}}$ech cohomology group. 
See \cite[Claim 1]{Fuj12-A} for the construction of $\phi_{i}$. 
The blew equality is obtained from $\I{h_{\e}}=\I{h}$. 
(In this step, we essentially use property $(c)$.) 
The cohomology class $\alpha=\{u\}$ represented by 
$u \in L^{n, q}_{(2)}(X, F)_{h, \omega}$ goes to the zero class by $j$ 
since we have $u= \dbar f$ for some $f \in L^{n, i-1}_{(2)}(Y, F)_{h_{\e_{0}}, \ome}$. 
By chasing the above diagram, we can obtain a solution 
$g \in L^{n, i-1}_{(2)}(Y, F)_{{h}, \omega}$ satisfying 
the $\dbar$-equation $u = \dbar g$. 

\qed
\vspace{0.4cm}

\subsection{Proof of Corollary \ref{main-co}}
In this subsection,  
we prove that $h_{\min}$ and $h_{\rm{Siu}}$ satisfy
the assumptions of Theorem \ref{main}.

First we recall the definition of Siu's metrics $h_{\rm{Siu}}$ 
(which was first introduced and 
plays an central role in  \cite{Siu98}). 
Let $F$ be a line bundle whose Kodaira dimension is non-negative 
(that is, some positive multiple of $F$ admits sections).  
For every positive integer $m>0$, let 
$h_{m}$ be the singular metric on $F$ 
induced by a basis of $H^{0}(X, F^m)$, 
namely $h_{m}$ is (locally) defined by 
\begin{equation*}
- \log h_{m}:= \frac{1}{2m} \log
\Big( \sum_{j=1}^{N_{m}} |s_{j}^{(m)}|^{2}  \Big), 
\end{equation*}
where 
$\{s_{j}  \}_{j=1}^{N_{m}}$ is a basis of $H^{0}(X, F^m)$. 
Then by taking  positive numbers 
$\{ \varepsilon_{m} \}_{m\geq 1}$  
Siu's metric $h_{\rm{Siu}}$ can be defined  by
\begin{equation*}
-\log h_{\rm{Siu}} = \log
\sum_{m \geq 1} \varepsilon_{m} \frac{1}{h_{m}}.    
\end{equation*}
This metric $h_{\rm{Siu}}$ and the multiplier ideal 
sheaf $\I{h_{\rm{Siu}}}$
depend on the choice of 
$\{\e_{k}\}_{k=1}^{\infty}$ (see \cite{Kim14}), 
but $h_{\rm{Siu}}$ always  
admits an analytic Zariski decomposition. 
Thus we can expect a Nadel vanishing theorem  
for $\I{h_{\rm{Siu}}}$.

\begin{cor}[=Theorem \ref{m-min}]\label{main-co}
Let $F$ be a big line bundle 
on a smooth projective variety $X$ and let 
$h$ be either a metric 
with minimal singularities $h_{\min}$  or 
Siu's metric $h_{\rm{Siu}}$  on $F$. 
Then we have 
\begin{equation*}
H^{i}(X, K_{X} \otimes F \otimes \I{h}) = 0 
\hspace{0.4cm} {\text{for}}\ {\text{any}}\ i >0.
\end{equation*}
\end{cor}
\begin{proof}
It is sufficient to prove that 
the metrics $h_{\min}$ and $h_{\rm{Siu}}$ satisfy 
the assumptions in Theorem \ref{main}. 
Since $F$ is big, 
some positive multiple $F^{m}$ of $F$ can be written 
as $F^{m}= A \otimes E$, 
where $A$ is a very ample line bundle and $E$ is an 
effective line bundle. 
Take a (non-zero) section $s_{A}$ (resp. $s_{E}$) 
of $A$ (resp. E). 
Let $h_{A}$ be the metric induced by a basis of 
$H^{0}(X, A)$. 
Then $h_{A}$ is smooth and has strictly positive curvature
since $A$ is very ample. 
Further let 
$h_{E}$ be the singular metric  defined by the section $s_{E}$. 
Now we consider the section $s$ of $F^m$ and 
the singular metric $h_{\rbig}$ on $F$ defined by 
\begin{equation*}
s := s_{A} \otimes s_{E}\hspace{0.2cm} {\rm{and}} \hspace{0.2cm}
h_{\rbig}: = (h_{A} h_{E})^{1/m}. 
\end{equation*} 
Since the curvature of $h_{\rbig}$ is equal to 
$(\sqrt{-1}\Theta_{h_{A}}(A) + [{\rm{div}}s_{E}])/m $,  
the metric $h_{\rbig}$ 
has strictly positive curvature. 
Further 
the pointwise norm $|s|_{h^{m}_{\rbig}}$ is bounded. 
This is because, 
the norm $|s|_{h^{m}_{\rbig}}$ is equal to 
$|s_{A}|_{h_{A}}$ and 
$h_{A}$ is a smooth metric.  
It remains to confirm assumption (2) in Theorem \ref{main}.
When $h$ is a metric with minimal singularities, 
assumption (2) is obvious by the definition of $h_{\min}$. 
We consider the case where $h$ is Siu's metric $h_{\rm{Siu}}$. 
By the construction of $h_{\rm{Siu}}$, 
the metric $h_{\rm{Siu}}$ is less singular than $h_{m}$. 
When $\{ t_{j} \}_{j=1}^{N}$ is a basis of $H^{0}(X, A)$, 
sections $\{ s_{E}\otimes t_{j} \}_{j=1}^{N}$ become a part of 
a basis of $H^{0}(X, F^{m})$. 
Therefore, by the construction of $h_{m}$ and $h_{\rbig}$, 
the metric $h_{m}$ is less singular than $h_{\rbig}$. 
\end{proof}

\section{Birational Zariski Decomposition and Minimal Singularities}
In this section, 
we show that 
metrics with minimal singularities do 
not always have 
analytic singularities even 
if line bundles are big. 
The content of this section is known to specialists, 
but we give the references and proof for the 
reader's 
convenience. 
Throughout this section, let $X$ be a compact 
K\"ahler manifold and $D$ be a divisor on $X$. 
We denote by $\mathcal{O}_{X}(D)$,  
the line bundle (the invertible sheaf) 
defined by $D$. 

\begin{prop}\label{birat}
Let $h_{\min}$ be a metric 
with minimal singularities on $\mathcal{O}_{X}(D)$.
Assume that $h_{\min}$ has analytic singularities. 
Then $\mathcal{O}_{X}(D)$ admits a birational Zariski decomposition. 
That is, there exist a modification 
$\pi: \widetilde{X} \to X$, 
a nef $\mathbb{R}$-divisor $P$ 
and an $\mathbb{R}$-effective divisor $N$ on $\widetilde{X}$ with the following properties$:$ 
\begin{itemize}
\item[$\bullet$] $\pi^{*}D = P + N$. 
\item[$\bullet$] For any positive integer $k > 0$, the map 
\begin{equation*}
H^{0}(X, \mathcal{O}_{\widetilde{X}}(\lfloor kP \rfloor))
\to 
H^{0}(X, \mathcal{O}_{\widetilde{X}}( kD ))
\end{equation*}
induced by the section $e_{k}$ is an isomorphism,  
where $e_{k}$ is the natural section of $\lceil kN \rceil$.
\end{itemize}
Here $\lfloor G \rfloor$ 
$($resp. $\lceil G \rceil$$)$ 
denotes the divisor defined by the
round-downs $($resp. the round-ups$)$ of the coefficients 
of an $\mathbb{R}$-divisor $G$. 
\end{prop}
\begin{proof}
Since $h_{\min}$ has analytic singularities, 
we can take an ideal sheaf 
$\mathcal{I} \subset \mathcal{O}_{X}$ such that 
a local weight $ \varphi _{\min}$ can be   
written as 
\begin{equation*}
\varphi_{\min} = \frac{c}{2} \log \big( 
|f_{1}|^{2} + |f_{2}|^{2} + \cdots + |f_{k}|^{2}\big) +v, 
\end{equation*}
where $f_{i}$ are local generators of $\mathcal{I}$   
(see Definition \ref{s-met}).  
By taking a resolution of $\mathcal{I}$, we obtain 
a modification $\pi : \widetilde{X} \to X$ with  
\begin{equation*}
\sqrt{-1}\Theta_{\pi^{*}h_{\min}} (\pi^{*}D) = 
\beta + [N],   
\end{equation*}
where $\beta $ is a smooth semi-positive $(1,1)$-form and 
$N$ is an effective $\mathbb{R}$-divisor. 
On the other hand, 
the pull-back $\pi^{*}h_{\min}$ of $h_{\min}$
is a metric with minimal singularities on 
$\mathcal{O}_{\widetilde{X}}(\pi^{*}D)$. 
Indeed, we fix a smooth metric $h_{0}$ on 
$\mathcal{O}_{X}(D)$ and 
take an arbitrary singular metric $g$ on $\mathcal{O}_{\widetilde{X}}(\pi^{*}D)$ 
with semi-positive curvature.  
Then there exists an $L^{1}$-function $\varphi_{g}$
on $\widetilde{X}$ such that 
$g= \pi^{*}h_{0} e^{-\varphi_{g}}$   
since $\widetilde{X}$ is also a compact K\"ahler manifold. 
Now we define the function $\psi_{g}$ on $X$ by     
\begin{equation*}
\psi_{g}(x):= \sup_{y \in \pi^{-1}(x)} \varphi_{g}(y)
\quad \text{for any}\ x\ \in X. 
\end{equation*}
Then the curvature of 
$h_{0} e^{-\psi_{g}}$ is also positive. 
Hence $h_{\min} \leq C h_{0} e^{-\psi_{g}}$ for 
some positive constant $C>0$ by the definition of 
$h_{\min}$. 
It yields $\pi^{*} h_{\min} \leq C \pi^{*}h_{0} e^{-\varphi_{g}} =C g$. 
Therefore $\pi^{*} h_{\min}$ is a metric with minimal singularities.

Put $P:= \pi^{*}D- N$. 
Then the first Chern class of $P$ contains the semi-positive 
form $\beta$ (in particular $P$ is nef). 
Finally we show that the map in Proposition \ref{birat} 
is isomorphic. 
For an arbitrary section 
$s \in H^{0}(X, \mathcal{O}_{X}( kD ))$,  
the metric $\pi^{*}h_{\min}^{k}$ is less singular than 
the metric $h_{s}$ induced by the section $s$, 
since $\pi^{*}h_{\min}^{k}$ is 
a metric with minimal singularities 
$\mathcal{O}_{\widetilde{X}}(k \pi^{*}D)$.   
In particular, the Lelong number of a weight of $k \pi^{*}h_{\min}$ 
is less than or equal to that of $h_{s}$, 
thus we obtain $kN \leq {\rm{div}} s $. 
Since ${\rm{div}} s$ is a $\mathbb{Z}$-divisor, 
we have $\lceil kN \rceil \leq {\rm{div}} s $. 
It implies that $s/e_{k}$ is a 
(holomorphic) section. 
This completes the proof. 
\end{proof}

In general, it is difficult to compute 
metrics with minimal singularities. 
Proposition \ref{birat} says 
that metrics with minimal singularities on line bundles  
admitting no birational Zariski decomposition 
never have analytic singularities. 
However there are at least two known examples which tell us that 
line bundles do not admit a birational Zariski decomposition, even if the line bundles are big.
(see \cite[IV, \S 2.6, Example 6.4]{Nak}, 
\cite[Theorem 1.1.]{Les12}).



\begin{thebibliography}{n}


\bibitem[Ber13]{Ber13}
B. Berndtsson.   
\textit{The openness conjecture for plurisubharmonic functions.}
Preprint, arXiv:1305.5781v1.


\bibitem[DEL00]{DEL00}
J.-P. Demailly, L. Ein, R. Lazarsfeld.  
\textit{A subadditivity property of multiplier ideals.}
Michigan Math. J. {\bf 48} (2000), 137--156.


\bibitem[Dem]{Dem}
J.-P. Demailly.   
\textit {Analytic methods in algebraic geometry.}
Compilation of Lecture Notes on the web
page of the author.


\bibitem[Dem-book]{Dem-book}
J.-P. Demailly.   
\textit {Complex analytic and differential geometry.}
Lecture Notes on the web page of the author.


\bibitem[Dem82]{Dem82}
J.-P. Demailly.  
\textit{Estimations $L^{2}$ pour l'op\'erateur $\overline{\partial}$ d'un 
fibr\'e vectoriel holomorphe semi-positif au-dessus d'une vari\'et\'e k\"ahl\'erienne compl\`ete.}
Ann. Sci. \'Ecole Norm. Sup(4). {\bf{15}} (1982), 457--511. 


\bibitem[DK01]{DK01}
J.-P. Demailly, J. Koll\'ar.  
\textit{Semicontinuity of complex singularity exponents and K\"ahler-Einstein metrics on Fano orbifolds.}
Ann. Sci. \'Ecole Norm. Sup(4) {\bf {34}} (2001), 525--556.



\bibitem[DPS01]{DPS01}
J.-P. Demailly, T. Peternell, M. Schneider.  
\textit{Pseudo-effective line bundles on compact K\"ahler manifolds.}
International Journal of Math {\bf{6}} (2001), 689--741. 



\bibitem[Eno90]{Eno90}
I. Enoki. 
\textit{Kawamata-Viehweg vanishing theorem for compact 
K\"ahler manifolds.}
Einstein metrics and Yang-Mills connections (Sanda, 1990), 
59--68. 


\bibitem[Fuj12]{Fuj12-A}
O. Fujino. 
\textit{A transcendental approach to Koll\'ar's injectivity theorem.}
Osaka J. Math. {\bf{49}} (2012), no. 3, 833--852.


\bibitem[GM13]{GM13}
Y. Gongyo, S. Matsumura. 
\textit{Versions of injectivity and extension theorems.}
Preprint, arXiv:1406.6132v2. 


\bibitem[GZ14]{GZ13}
Q. Guan, X. Zhou.
\textit {Effectiveness of 
Demailly's strong openness conjecture and related problems.}
to appear in Invent. Math.


\bibitem[Hie14]{Hie14}
P. H. Hiep
\textit{The weighted log canonical threshold.}
C. R. Math. Acad. Sci. Paris {\bf{352}} (2014), no. 4, 283--288.


\bibitem[Kaw82]{Kaw82}
Y. Kawamata.
\textit{A generalization of Kodaira-Ramanujam's vanishing theorem.}
Math. Ann. {\bf{261}} (1982), no. 1, 43--46.



\bibitem[Kim14]{Kim14}
D. Kim. 
\textit {Equivalence of plurisubharmonic singularities and Siu-type metrics.}
Preprint, arXiv:1407.6474v2. 


\bibitem[Kol86]{Kol86}
J. Koll\'ar. 
\textit{Higher direct images of dualizing sheaves. I.}
Ann. of Math. (2) {\bf{123}} (1986), no. 1, 11--42.


\bibitem[Laz]{Laz}
R. Lazarsfeld.  
\textit{Positivity in Algebraic Geometry I-II.}
Springer Verlag, Berlin, (2004).


\bibitem[Les12]{Les12}
J. Lesieutre.
\textit{The diminished base locus is not always closed.}
Compos. Math. {\bf{150}} (2014), no. 10, 1729--1741. 




\bibitem[Mat13]{Mat13-B}
S. Matsumura. 
\textit{An injectivity theorem with multiplier ideal sheaves of singular metrics with transcendental singularities.}
Preprint, arXiv:1308.2033v2.


\bibitem[Mat14A]{Mat14A}
S. Matsumura. 
\textit{
A Nadel vanishing theorem via the injectivity theorem.}
Math. Ann. {\bf{359}} (2014), no. 3-4, 785--802.

\bibitem[Mat14B]{Mat14B}
S. Matsumura. 
\textit{Injectivity theorems with multiplier ideal sheaves and their applications.}
to appear in Recent developments in Complex Analysis and Geometry, 
the Proceeding of the 10-th Korean Conference on Several Complex Variables, 
Springer-Verlag. 



\bibitem[Nad89]{Nad89}
A. M. Nadel. 
\textit{Multiplier ideal sheaves and existence of 
K\"ahler-Einstein metrics of positive scalar curvature.}
Proc. Nat. Acad. Sci. U.S.A. {\bf{86}} (1989), 
no. 19, 7299--7300. 


\bibitem[Nad90]{Nad90}
A. M. Nadel. 
\textit{Multiplier ideal sheaves and K\"ahler-Einstein metrics of positive scalar curvature.}
Ann. of Math. (2) {\bf{132}} (1990), no. 3, 549--596. 


\bibitem[Nak]{Nak}
N. Nakayama. 
\textit{Zariski-decomposition and abundance.} 
MSJ Memoirs, {\bf{14}}. Mathematical Society of Japan, Tokyo, (2004).



\bibitem[Ohs04]{Ohs04}
T. Ohsawa.
\textit{On a curvature condition that implies a cohomology injectivity theorem of Koll\'ar-Skoda type.} 
Publ. Res. Inst. Math. Sci. {\bf{41}} (2005), 
no. 3, 565--577.


\bibitem[P$\rm{\check{a}}$u12]{Pau12}
M. P$\rm{\check{a}}$un. 
\textit{Relative critical exponents, non-vanishing and 
metrics with minimal singularities.}
Invent. Math. {\bf{187}} (2012), no. 1, 195--258.  


\bibitem[Siu98]{Siu98}
Y.-T. Siu.
\textit{Invariance of plurigenera.} 
Invent. Math. {\bf{134}} (1998), no. 3, 661--673.


\bibitem[Vie82]{Vie82}
E. Viehweg. 
\textit{Vanishing theorems.}
J. Reine Angew. Math. {\bf{335}} (1982), 1--8. 

\end{thebibliography}
\end{document}